\documentclass[reqno]{amsart}

\usepackage{amssymb,amscd,color}
\usepackage{pdfsync}
\usepackage{graphics}
\usepackage{graphicx}

\theoremstyle{plain}

\newtheorem*{thm}{Theorem}

\newtheorem{theorem}{Theorem}[section]

\newtheorem{lemma}[theorem]{Lemma}
\newtheorem{proposition}[theorem]{Proposition}

\theoremstyle{definition}

\newtheorem{remark}{Remark}

\newcommand{\N}{\hbox{\ensuremath{\mathbb{N}}}}
\newcommand{\Z}{\hbox{\ensuremath{\mathbb{Z}}}}
\newcommand{\nZ}{\frac{1}{n}\mathbb{Z}}

\newcommand{\R}{\hbox{\ensuremath{\mathbb{R}}}}

\newcommand{\beq}{\begin {equation}}
\newcommand{\eeq}{\end{equation}}
\newcommand{\essinf}{{\rm ess\, inf}\ }
\newcommand{\esssup}{{\rm ess\,sup}\ }

\numberwithin{equation}{section}
\begin{document}


\title[Balian-Low type Theorems ]{Uncertainty Principles and Balian-Low type Theorems in Principal Shift-Invariant Spaces }

\author[Akram Aldroubi]{Akram~Aldroubi} \address[Akram Aldroubi]{Department of Mathematics\\
Vanderbilt University\\ 1326 Stevenson Center\\ Nashville, TN 37240}
\email 
{akram.aldroubi@vanderbilt.edu}

\author[Qiyu Sun]{Qiyu~Sun}
\address[Qiyu Sun]{Department of Mathematics\\
University of Central Florida\\ Orlando, FL 32816}
\email
{qsun@mail.ucf.edu}

\author[Haichao Wang]{Haichao~Wang}
\address[Haichao Wang]{Department of Mathematics\\
Vanderbilt University\\ 1326 Stevenson Center\\ Nashville, TN 37240}
\email 
{haichao.wang@vanderbilt.edu}


\thanks{ The research of A.~Aldroubi is supported in part by NSF Grant DMS-0807464.}

\maketitle

\begin{abstract}
In this paper, we consider the time-frequency localization of the  generator of a  principal shift-invariant space on the  real line which has additional shift-invariance.
 We prove that if a principal shift-invariant space  on the real line is  translation-invariant then
 any of its orthonormal (or Riesz) generators is non-integrable. However, for any $n\ge2$, there exist principal shift-invariant spaces on the real line that are also $\nZ$-invariant  with an integrable orthonormal (or a Riesz) generator $\phi$, but $\phi$ satisfies
 $\int_{\mathbb R} |\phi(x)|^2 |x|^{1+\epsilon} dx=\infty$ for any $\epsilon>0$ and
 its Fourier transform  $\widehat\phi$ cannot decay as fast as $ (1+|\xi|)^{-r}$ for any $r>\frac{1}{2}$.
Examples are constructed to demonstrate that
the above decay properties  for the orthormal generator in the time domain and  in the frequency domain are optimal.
\end{abstract}


\section{Introduction and Main Results}

In this paper, a {\em principal shift-invariant space} on the real line  is a shift-invariant space $V_2(\phi)$ generated by a function
 $\phi\in L^2:= L^2({\mathbb R})$,
\begin{equation}\label{sis.def}
V_2(\phi):=\Big\{\sum_{k\in {\mathbb Z} }c(k)\phi(\cdot-k)\big|\ c:=(c(k))_{k\in \mathbb Z}\in \ell^2:=\ell^2({\mathbb Z})\Big\},\end{equation}
such that $\{\phi(\cdot-k)|\ k\in {\mathbb Z}\}$ is a Riesz basis for $V_2(\phi)$, i.e.,
there exist positive constants $A$ and $B$ such that
\begin{equation}\label{riesz.def}
A\|c\|_{\ell^2} \le\Big \|\sum_{k\in {\mathbb Z} }c(k)\phi(\cdot-k)\Big\|_2\le B\|c\|_{\ell^2}\quad {\rm for \ all}
\ c:=(c(k))_{k\in {\mathbb Z}}\in \ell^2.
\end{equation}
The function $\phi$ is called the {\em generator} of the principal shift-invariant space $V_2(\phi)$, and it is called the {\em orthonormal generator}
if
 $\{\phi(\cdot-k)|\ k\in {\mathbb Z}\}$ is an orthonormal basis for $V_2(\phi)$, i.e., \eqref{riesz.def} holds for $A=B=1$.
Principal shift-invariant spaces have been widely used in  approximation theory,  numerical analysis, sampling theory and
 wavelet theory (see e.g., \cite{AG01, AST05, Bow00,  CS07, UB00} and the references therein).

The classical models of principal shift-invariant spaces on the real line are the Paley-Wiener space $PW$ also known as the space of bandlimited functions (the set of all square-integrable
 functions  bandlimited to $[-1/2, 1/2]$) and the  spline space $S_n^{n-1}$ (the set of all $(n-1)$-differentiable
 square-integrable functions whose restriction on any integer interval $[k,k+1]$ coincides with a polynomial of degree at most $n$).
 More precisely, the  Paley-Wiener space $PW$ is  the shift-invariant space
 generated by the sinc function
 ${\rm sinc} (x)=\frac{\sin \pi x}{\pi x}$, i.e. $PW=V_2({\rm sinc} )$
 and
 the spline space $S^{n-1}_n$ is generated by the $B$-spline $\beta^n$, i.e. $S_n^{n-1}=V_2(\beta^n)$
where $\beta^0$ is the characteristic function on $[0,1)$ and $\beta^n, n\ge 2$, are defined iteratively by
$\beta^n(t)=\int_{\small\R} \beta^{n-1}(t-\tau)\beta^0(\tau) d\tau$.

\bigskip
Now we consider principal shift-invariant spaces that are  invariant under additional  set of translates other than $\Z$. The shift-invariant spaces with additional invariance have been used in the study of
wavelet analysis and sampling theory \cite{Web00,CS03,HL09}, and have been completely characterized in \cite {ACHKM10} for $L^2(\R)$ and in \cite {ACP09} for $L^2(\R^n).$ For a   subspace $V$ of $L^2({\mathbb R})$, let
\begin{equation}\label{tauv.def}
\tau(V):=\Big\{ t\in {\mathbb R}\big|  \ f(\cdot-t) \ {\rm belong \ to} \
 V \ {\rm for\ all} \ f\in V\Big\}.
\end{equation}
  For any closed subspace $V$ of $L^2$, one may verify that $\tau(V)$ is a closed additive subgroup of ${\mathbb R}$, and hence
  $\tau(V)$ is either $\{0\}$, or ${\mathbb R}$, or $\alpha {\mathbb Z}$ for some $\alpha>0$.
It can be shown that  \cite{ACHKM10} for a principal shift-invariant space $V_2(\phi)$ on the real line
 \begin{equation}
\tau(V_2(\phi))={\mathbb R} \ {\rm or} \ \tau(V_2(\phi))=\nZ \  \ {\rm for \ some} \  n\in  {\mathbb N}.
\end{equation}
We say  that a  shift-invariant space $V$ on the real line
 has additional  invariance if $\tau (V)\supsetneq {\mathbb Z}$.
 It is well-known that
 the Paley-Wiener  space $PW$   are invariant under all translations. Thus,
\begin{equation*}
\tau(PW)={\mathbb R}.
\end{equation*}
A closed subspace $V$ of $L^2$ with $\tau(V)={\mathbb R}$ is usually known as a {\em translation-invariant space}. The fact that the space of bandlimited functions $PW$ is  translation-invariant ($\tau(PW)={\mathbb R}$) makes it useful for modeling  signals and  images.  However,  it is known that any function $\phi$ that generates a Riesz basis for $PW$ has slow spatial-decay in the sense that $\phi\notin L^1(\R)$, e.g., ${\rm sinc} (x)=\frac{\sin \pi x}{\pi x}.$ This slow spatial-decay property for the generator of principal shift-invariant spaces $V_2(\phi)$ that are also translation-invariant is not unique to the space of bandlimited functions $PW.$  In fact, in this paper, we first show that the generator $\phi$ of any translation-invariant principal shift-invariant space $V_2(\phi)$ on the real line is not integrable.

\begin{theorem} \label{bandlimited.tm}
Let $\phi\in L^2$ and $\{\phi(\cdot-k)|\ k\in {\mathbb Z}\}$ be a Riesz basis for its generating space $V_2(\phi)$.
If $V_2(\phi)$ is translation-invariant then $\phi\not\in L^1:=L^1({\mathbb R})$.
\end{theorem}
The slow spatial-decay of the generators of shift-invariant spaces that are also translation-invariant is a disadvantage for the numerical implementation of some analysis and processing algorithms.

On the other hand, Riesz bases for the spline spaces $S_n^{n-1}=V_2(\beta^n)$ can be generated by the compactly supported B-spline functions $\beta^n.$ This is one of the reasons that spline spaces   are often used in signal and image processing algorithms as well as in numerical analysis.  Moreover, the B-spline functions $\beta^n$ are also well localized in frequency domain, since $\hat \beta^n(\xi)=O(|\xi|^{-n-1})$. However, the spaces $S_n^{n-1}=V_2(\beta^n)$ have no invariance other than  by integer shifts. In fact, it can be shown that any principle shift-invariant space $V_2(\phi)$ generated by a compactly supported function $\phi$ cannot have any invariance other than by integer shifts \cite {HL09, ACHKM10}.

One way to circumvent some of the problems is to seek  principle shift-invariant spaces $V_2(\phi)$ that are close to being translation invariant, with a  generator $\phi$ which is well localized in both space and frequency domains, i.e., $\phi$ and $\widehat \phi $ are well localized. Specifically, we ask whether we can find a shift-invariant space $V(\phi)$ such that $V(\phi)$ is also $\nZ$  invariant for some $2\le n\in \N $, and such that $\phi$ and $\widehat \phi$ are well localized. It turns out that it is possible to construct functions $\phi$ that are well-localized in time and frequency domains, that generate shift-invariant spaces $V_2(\phi)$ that are also $\nZ$ invariant. However, there are uncertainty and Balian-Low type obstructions, as will be described below. Specifically, the classical uncertainty principle tells us that there is a lower limit on the simultaneous time-frequency localization  of functions as shown by

\begin{thm}[Uncertainty Principle]
For any function $f\in L^2(\mathbb R)$, we have
\begin{equation}\label{UP}
||f||^2_2\leq 4\pi||xf(x)||_2||\xi\widehat{f}(\xi)||_2,
\end{equation}
and the equality holds only if
$$f(x)=ce^{-sx^2}$$
for $s>0$ and $c\in\R$.
\end{thm}

If we impose more conditions, the time-frequency localization deteriorates even further (see e.g., \cite {BCGP03, BCPS06,BHW95,CP07,Gau09,GH04,GHHK02,HP06} and the references therein).  For example, if the Gabor system $\{E_mT_ng\}_{m,n\in\Z}=\{e^{2\pi imx}g(x+n)\}_{m,n\in\Z}$ of a function $g$ is a Riesz basis for $L^2(\mathbb R)$, we  will  have the following Balian-Low theorem:
\begin{thm}[Balian-Low]
Let $g\in L^2(\mathbb R)$. If $\{E_mT_ng\}$ is a Riesz basis for $L^2(\mathbb R)$, then
$$\Big(\int_{-\infty}^{\infty}|xg(x)|^2dx\Big)\Big(\int_{-\infty}^{\infty}|\xi\widehat{g}(\xi)|^2\Big)d\xi=\infty.$$
\end{thm}

\smallskip

The Balian-Low theorem implies that if function $g$ generates a Gabor Riesz basis, then it is  not possible for the functions $g$ and $\widehat{g}$ to be simultaneously well-localized. In particular
$$|g(x)|<\frac{c}{|x|^r}, \ \ |\widehat{g}(\xi)|<\frac{c}{|\xi|^r}$$
cannot hold simultaneously with $r> 3/2$.

\smallskip
\subsection{Balian-Low type results for shift-invariant spaces}
For the case of a shift-invariant space $V_2(\phi)$ which is also $\nZ$ invariant for some $2\le n\in\N$, we  obtain the following surprising result:
\begin {theorem}
\label{GUP}
 If $\phi\in L^2$ has the property that
 $\{\phi(\cdot-k)|\ k\in {\mathbb Z}\}$ is a Riesz basis for its generating space $V_2(\phi)$, and
that $V_2(\phi)$ is $\nZ$-invariant for some $n\ge2$, then for any $\epsilon>0$, we have
\begin{equation}\label{time.tm.eq1} \int_{\mathbb R} |\phi(x)|^2 |x|^{1+\epsilon} dx=+\infty.
\end{equation}
\end{theorem}

\bigskip

\begin {remark} {$ $}\\

\begin{enumerate}

\item [(i)] Theorem \ref {GUP} is a Balian-Low type result. If we choose $\epsilon=1$ in \eqref{time.tm.eq1} of Theorem \ref {GUP},  we get $\int_{\mathbb R} |x\phi(x)|^2 dx=+\infty.$ It should be noted that in the Balian-Low Theorem $\int_{-\infty}^{\infty}|xg(x)|^2dx$ can be finite, while in the case of Theorem \ref {GUP} $\int_{\mathbb R} |x\phi(x)|^2 dx$ is always infinite. For the case $\Delta_p= \int_{\mathbb R} |\phi(x)|^2 |x|^{p} dx$, the theorem above should be comparable to the $(1,\infty)$ version of the Balian-Low  Theorem (\cite {BCPS06}, \cite{Gau09}).
\item [(ii)] If we do not require  other invariances besides integer shifts, then we can find $V_2(\phi)$ such that  $\{\phi(\cdot-k): \; k \in \Z\}$  is an orthonormal basis for $V$ and such that $\phi$ decays exponentially  in both time and frequency. In particular for such a $\phi$ it is obvious that $\Big(\int_{-\infty}^{\infty}|x|^\alpha |g(x)|^2dx\Big)\Big(\int_{-\infty}^{\infty}|\xi|^\beta |\widehat{g}(\xi)|^2d\xi\Big)<\infty,$ where $\alpha,\beta>0$ are any positive real numbers.
\end{enumerate}
\end{remark}

\bigskip

There is also a decay restriction in the Fourier domain.  Specifically, the Fourier transform of an integrable generator $\phi$ of a principal shift-invariant space which is $\nZ$-invariant
for some integer  $n\ge 2$ cannot decay faster than $|\xi|^{-1/2-\epsilon}$ for any $\epsilon>0.$

\begin{theorem}\label{fGUP}
 Let $2\le n\in {\mathbb N}$.  Let $\phi \in L^1\cap L^2$ have the property that
 $\{\phi(\cdot-k)|\ k\in {\mathbb Z}\}$ is a Riesz basis for its generating space $V_2(\phi)$,
and that $V_2(\phi)$ is $\nZ$-invariant,  then
  for any $\epsilon>0$,
\begin{equation}\label{frequency.tm.eq1}\sup_{\xi\in\mathbb R} |\widehat\phi(\xi)| |\xi|^{1/2+\epsilon} =+\infty.
\end{equation}
\end{theorem}

We conclude from  Theorem \ref {fGUP} that there is an obstruction to pointwise frequency (non)-localization property.
\begin {remark}
The conclusion of  Theorem \ref {fGUP}  remains valid if we weaken the condition that $\phi \in L^1\cap L^2$  to $\phi \in L^2$ and $\widehat \phi$ is continuous.
\end {remark}

\subsection{ Optimality of the Balian-Low type results}Now, we show the optimality of the results of Theorems \ref{GUP} and \ref {fGUP}.

\smallskip

The following result shows that \eqref {frequency.tm.eq1} in  Theorem \ref {fGUP} is sharp and that for any $2\le n\in \N$ there exists a  generator $\phi \in L^1\cap L^2$ (that depends on $n$) for  $V_2(\phi)$ such that $\widehat \phi$ decays like $|\xi|^{-1/2}$. This is done by constructing time-frequency localized generators $\phi$ that achieve the desired properties:

\begin {theorem}
\label{PointwiseFreqLoc}
For each integer $n\ge 2$, there exists a function $\phi \in L^1\cap L^2$  (and hence $\widehat \phi$ is continuous) which depends on $n$,  such that
$\{\phi (\cdot-k)|\ k\in {\mathbb Z}\}$ is an orthonormal basis for its generating space $V_2(\phi)$,
$V_2(\phi)$ is $\nZ$-invariant,    and

\begin {equation}\int_{\mathbb R}   |\phi(x)|^2 (1+|x|)^{1-\epsilon}dx<\infty,\end{equation}
\begin{equation}
\label{frequency.tm.eq2} \sup_{\xi\in\mathbb R}  |\widehat\phi(\xi)| |\xi|^{1/2}<+\infty.\end{equation}

\end {theorem}
\begin {remark} ${}$

\begin {enumerate}
\item [(i)]Note that by giving up the translation invariance and only allowing $1/n$ invariance as in Theorem \ref {PointwiseFreqLoc}, we are able to have an $L^1$ generator, while this is not possible for translation invariance as shown in Theorem \ref {bandlimited.tm}.
\item [(ii)] Note that Theorem \ref {PointwiseFreqLoc} shows the optimality of both Theorems \ref  {GUP} and \ref {fGUP} simultaneously.
\end {enumerate}
\end{remark}

We now turn our attention to the integral measure of time-frequency localization, and show that  \eqref {time.tm.eq1} in Theorem \ref {GUP} is nearly optimal.

\begin{theorem}\label{time.tm}  For any $2\le n\in \N$, $\epsilon>0$,  $\gamma \ge0,$ $\delta>0$, $1\le q<\infty$  with $1+\delta-q/2<1/(2\gamma),$ there exists  $\phi \in L^2$ (that depends on $\epsilon,\delta, q,\gamma, n$) such that
$\{\phi (\cdot-k)|\ k\in {\mathbb Z}\}$ is an orthonormal basis for its generating space $V_2(\phi)$,
$V_2(\phi)$ is $\nZ$-invariant and $\phi$ satisfies the following conditions:
\begin{enumerate}
\item $\int_{\mathbb R}   |\phi(x)|^2 (1+|x|)^{1-\epsilon}dx<\infty,$

\item $\int_{\mathbb R}   |\phi(x)| (1+|x|)^\gamma dx<\infty,$

\item $\int_{\mathbb R}   |\widehat\phi(\xi)|^q (1+|\xi|)^{\delta} d\xi<\infty.$

\end{enumerate}
\end{theorem}

\begin {remark} ${ }$

 \begin {itemize}

 \item [(i)] Note that the orthonormal generator $\phi=sinc$ for the Paley-Wiener space $PW$ satisfies the first and third localization properties in Theorem \ref {time.tm}. This shows that Theorem \ref {GUP} is optimal.  However, the $sinc$ function does not satisfies the second time localization inequality. In fact no function $\phi$ generating a shift-invariant space $V_2(\phi)$ that is also translation invariant can satisfy the second inequality of  Theorem \ref {time.tm}, as is shown in Theorem \ref {bandlimited.tm}. Thus by relaxing translation invariance to $\nZ$ invariance we are able to get better time localization in the sense of the second localization inequality above. For this however, we needed to trade off some frequency localization by allowing infinite support in frequency.
\smallskip

 \item [(ii)] We do not know what happens for the  case $\epsilon=0.$

 \smallskip

 \item [(iii)] Using Lemmas \ref {TimeLem}, \ref {TimeLemp} and   \ref {FreqLem}, Theorem \ref {time.tm} can be shown to be  valid for other norms and other weights.
  \end {itemize}
\end{remark}

\section{Proofs}
\subsection{Proof of Theorem \ref{bandlimited.tm}}
To prove Theorem \ref{bandlimited.tm}, we  recall a characterization for the  Riesz (orthonormal) basis  property (see e.g., \cite {DDR94}) and for  the translation-invariance property (see \cite{ACHKM10}).

\begin{proposition}  \label{riesz.prop}
Let $\phi\in L^2$. Then
\begin{itemize}
\item[{(i)}] $\{\phi(\cdot-k)|\ k\in {\mathbb Z}\}$ is a Riesz basis for its generating space $V_2(\phi)$ if and only if
$$m\leq\sum_{k\in {\mathbb Z}} |\widehat{\phi}(\xi+k)|^2\leq M \ {\rm for \ almost \ all}\ \xi\in {\mathbb R}$$
where  $m$ and $M$ are positive constants, and
\item[{(ii)}] $\{\phi(\cdot-k)|\ k\in {\mathbb Z}\}$ is an orthonormal  basis for  its generating space $V_2(\phi)$ if and only if
$$\sum_{k\in {\mathbb Z}}|\widehat{\phi}(\xi+k)|^2=1 \ {\rm for \ almost \ all }\ \xi\in {\mathbb R}.$$
\end{itemize}
\end{proposition}

For shift-invariant spaces that are also translation invariant, the following proposition is a special case of a general result in \cite{ACHKM10}.
\begin{proposition}\label{sis.prop}  Let   $\phi\in L^2$ with the property that
 $\{\phi(\cdot-k)|\ k\in {\mathbb Z}\}$ is a Riesz basis for its generating space $V_2(\phi)$. Then $V_2(\phi)$ is
translation-invariant if and only if for almost all $\xi\in\R$,
$$\widehat \phi(\xi)\widehat \phi(\xi+k)=0 \ {\rm for\ all} \ 0\ne k\in {\mathbb Z}.$$
\end{proposition}

Now we start to prove Theorem \ref{bandlimited.tm}.

\begin{proof}[Proof of Theorem \ref{bandlimited.tm}]
Suppose on the contrary that there exists a principal shift-invariant space $V_2(\phi)$ on the real line
such that   $V_2(\phi)$ is translation-invariant and the generator $\phi$ is integrable.
Let $${\mathcal O}:=\big\{\xi\in {\mathbb R} |\ \widehat \phi(\xi)\ne 0\big\}.$$ Since $\phi\in L^1$ by assumption, $\widehat \phi$ is continuous, and hence ${\mathcal O}$ is an open set. From  Proposition \ref{sis.prop} it follows that
the Lebesgue measure  of  the set $({\mathcal O}+j)\cap ({\mathcal O}+k)$ is zero for all $j\ne k\in {\mathbb Z}$.
This together with the fact that ${\mathcal O}$ is an open set gives that
\begin{equation}\label{bandlimited.tm.pf.eq1} ({\mathcal O}+j)\cap ({\mathcal O}+k)=\emptyset \quad {\rm for \ all} \  j\ne k\in {\mathbb Z}.
\end{equation}
Recall that $\R$ is connected and that any connected set is not a union of nonempty disjoint open sets. Thus $ \{{\mathcal O}+k| k\in {\mathbb Z}\}$ is not an open covering of the real line, i.e.,  ${\mathbb R}\backslash (\cup_{k\in \mathbb Z} ({\mathcal O}+k))\ne \emptyset$, which in turn implies the existence of a real number $\xi_0\in {\mathbb R}$ with the property that
\begin{equation} \label{bandlimited.tm.pf.eq2}\widehat \phi(\xi_0+k)=0\quad  {\rm for \ all} \ k\in {\mathbb Z}.\end{equation}
As $\widehat \phi$ is uniformly continuous by  the assumption that $\phi\in L^1$, for any $\epsilon>0$ there exists $\delta>0$ such that
\begin{equation}\label{bandlimited.tm.pf.eq3}
|\widehat \phi(\xi+k)-\widehat \phi(\xi_0+k)|<\epsilon \quad {\rm for \ all} \ |\xi-\xi_0|< \delta \ {\rm and} \ k\in {\mathbb Z}.
\end{equation}
By \eqref{bandlimited.tm.pf.eq1}, for any $\xi\in {\mathbb R}$ there exists an integer $l(\xi)$ such that
 \begin{equation}\label{bandlimited.tm.pf.eq4} \sum_{k\in {\mathbb Z}}|\widehat{\phi}(\xi+k)|^2=|\widehat\phi (\xi+l(\xi))|^2.\end{equation}
  Combining  \eqref{bandlimited.tm.pf.eq2}, \eqref{bandlimited.tm.pf.eq3} and \eqref{bandlimited.tm.pf.eq4}  yields
 \begin{equation}
 \sum_{k\in {\mathbb Z}}|\widehat{\phi}(\xi+k)|^2< \epsilon^2\quad  {\rm whenever} \ |\xi-\xi_0|< \delta.
 \end{equation}
Since  $\epsilon>0$ can be   chosen to be arbitrarily small,
the last inequality contradicts the Riesz basis property that there exists $m>0$ such that $m\leq\sum_{k\in {\mathbb Z}} |\widehat{\phi}(\xi+k)|^2$ for almost all $\xi\in\R$.
\end{proof}

\bigskip
\subsection {Proof of Theorem \ref {GUP}}
We need a characterization of $\nZ$-invariance, which is a special case of a more general result in \cite{ACHKM10}.
\begin{proposition}\label{sis.prop2} {\rm (\cite{ACHKM10})}\ Let $n\ge 2$ be an integer, and  $\phi\in L^2$ with the property that
 $\{\phi(\cdot-k)|\ k\in {\mathbb Z}\}$ is  a Riesz basis for its generating space $V_2(\phi)$. Then $V_2(\phi)$ is
$\nZ$-invariant if and only if   for  almost all $\xi\in {\mathbb R}$, one and only one of the following
vectors
\begin{equation}\label{sis.prop2.eq1}
\Phi_m(\xi):=(\cdots, \widehat \phi(\xi+m-n), \widehat \phi(\xi+m), \widehat \phi(\xi+m+n), \cdots ),\quad  0\le m\le n-1,\end{equation} is nonzero.
\end{proposition}


\begin{proof} [Proof of Theorem \ref{GUP}]

Suppose on the contrary that
\begin{equation}\label{time.tm.pf.eq1}\int_{\mathbb R} |\phi(x)|^2  (1+ |x|)^{1+\epsilon}dx<\infty.\end{equation} Then $\phi\in L^1$, which implies that $\widehat \phi$ is a uniformly continuous function.
Let
 ${\mathcal O}_m=\{\xi\in {\mathbb R}|\ \Phi_m(\xi)\ne 0\}, 0\le m\le n-1$, where $\Phi_m$ is defined as in \eqref{sis.prop2.eq1}. Since
 $${\mathcal O}_m=\bigcup_{k\in\Z}\{\xi\in {\mathbb R}|\ \widehat{\phi}(\xi+m+kn)\ne 0\},$$
then
${\mathcal O}_m, 0\le m\le n-1$ are open sets, and
\begin{equation}\label{time.tm.pf.eq2}
{\mathcal O}_m+m={\mathcal O}_0 \ {\rm and} \ {\mathcal O}_m+nk={\mathcal O}_m \quad {\rm for \ all}\
 0\le m\le n-1 \ {\rm and} \ k\in {\mathbb Z}.
\end{equation}
Moreover, the intersection between  the sets ${\mathcal O}_m$  with different $m$ have zero Lebesgue measure (hence are  empty sets) by Proposition \ref{sis.prop2}.
Therefore $\{{\mathcal O}_m| 0\le m\le n-1\}$ is not an open covering of the real line ${\mathbb R}$, which implies that
the existence of a real number $\xi_0\in {\mathbb R}$ with the property that
\begin{equation}\label{time.tm.pf.eq3}
\widehat \phi(\xi_0+k)=0\quad {\rm for \ all}\  k\in {\mathbb Z}.
\end{equation}

Let $N\ge 1$ be a sufficiently large integer, $\delta=N^{-1-\epsilon/2}$, and $h$ be a smooth function supported on $[-2, 2]$ and satisfy $0\le h\le 1$, and $h(x)=1$  when $x\in [-1,1]$.
Define $\phi_N(x)= h(x/N) \phi(x)$.
Then we obtain that
\begin{eqnarray} \label{time.tm.pf.eq4}
& & \Big(\frac{1}{2\delta}\int_{-\delta}^\delta \sum_{k\in \mathbb Z} |(\widehat \phi-\widehat \phi_N)(\xi_0+\xi+k)|^2 d\xi\Big)^{1/2}\nonumber\\
& \le & \Big(\frac{1}{2\delta}\int_{\mathbb R} |(\widehat \phi-\widehat \phi_N)(\xi)|^2 d\xi\Big)^{1/2}=
\Big(\frac{1}{2\delta}\int_{\mathbb R} |\phi(x)-\phi_N(x)|^2 dx\Big)^{1/2}\nonumber\\
& \le &
 N^{-\epsilon/4} \Big(\int_{\mathbb R} |\phi(x)|^2 (1+|x|)^{1+\epsilon} dx\Big)^{1/2},
\end{eqnarray}
and
\begin{eqnarray}\label{time.tm.pf.eq5}
& & \Big(\frac{1}{2\delta}\int_{-\delta}^\delta \sum_{k\in \mathbb Z} |\widehat \phi_N(\xi_0+\xi+k)-\widehat \phi_N(\xi_0+k)|^2 d\xi\Big)^{1/2}\nonumber\\
&= &  \Big(\frac{1}{2\delta}\int_{-\delta}^\delta \sum_{k\in \mathbb Z}
 \Big|\int_0^\xi {\widehat \phi^\prime_N} (\xi_0+\xi'+k)d\xi'\Big|^2 d\xi\Big)^{1/2}\nonumber\\
 &\le&\Big(\frac{1}{2\delta}\int_{-\delta}^\delta \xi
 \int_0^\xi  \sum_{k\in \mathbb Z}\Big|\widehat \phi_N^\prime (\xi_0+\xi'+k)\Big|^2d\xi' d\xi\Big)^{1/2}\nonumber\\
 &\le&\Big(\frac{1}{2\delta}\int_{-\delta}^\delta \xi
 \int_0^\xi  \sum_{k\in \mathbb Z}\Big|\int_{\mathbb R} N^2|{\widehat h}^\prime(N\eta)| |\widehat \phi(\xi_0+\xi'+k-\eta)|d\eta \Big|^2d\xi' d\xi\Big)^{1/2}\nonumber\\
& \le   &  \Big(\frac{N^3}{2\delta}\|{\widehat h}^\prime\|_1 \int_{-\delta}^\delta \xi\int_0^\xi\int_{\mathbb R}|{\widehat h}^\prime(N\eta)| \Big(\sum_{k\in \mathbb Z} |\widehat \phi (\xi_0+\xi'+k-\eta)|^2\Big) d\eta d\xi'd\xi\Big)^{1/2}\nonumber\\
 & \le &    N^{-\epsilon/2} \|{\widehat h}^\prime\|_1  \Big(\esssup_{\xi\in {\mathbb R}}\sum_{k\in \mathbb Z} |\widehat \phi (\xi+k)|^2\Big)^{1/2},
\end{eqnarray}
where
$
\widehat \phi_N(\xi) =   N \int_{\mathbb R} {\widehat h}(N\eta) \widehat \phi(\xi-\eta) d\eta$
is used to obtain the second inequality, where the third inequality is obtained by letting $|\widehat h^\prime(N\eta)|=|\widehat h^\prime(N\eta)|^{1/2}|\widehat h^\prime(N\eta)|^{1/2}$ and using H\"older inequality. Also we have that
\begin{eqnarray} \label{time.tm.pf.eq6}
\sum_{k\in {\mathbb Z}} |{\widehat \phi}_N(\xi_0+k)|^2 & = &
\sum_{k\in {\mathbb Z}} \Big|\int_{\mathbb R} e^{-2\pi i (\xi_0+k) x}  \phi (x) (1- h(x/N)) dx\Big|^2 \nonumber\\
& \le & \int_0^1 \Big(\sum_{l\in {\mathbb Z}} |\phi(x+l)| |1-h((x+l)/N)|\Big)^2 dx\nonumber \\
& \le & \int_0^1 \Big(\sum_{l\in \mathbb Z} |\phi(x+l)|^2 (1+|x+l|)^{1+\epsilon}\Big) \nonumber\\
 & & \quad \times
\Big(\sum_{l\in \mathbb Z} (1-h((x+l)/N))^2 (1+|x+l|)^{-1-\epsilon}\Big)  dx\nonumber\\
& \le &  2 \Big(\sum_{l=N}^\infty |l|^{-1-\epsilon}\Big) \times \Big(\int_{\mathbb R}|\phi(x)|^2 (1+|x|)^{1+\epsilon} dx\Big),
\end{eqnarray}
where the first equality follows from \eqref{time.tm.pf.eq3}.
Combining \eqref{time.tm.pf.eq4},
\eqref{time.tm.pf.eq5} and \eqref{time.tm.pf.eq6} with Proposition \ref{riesz.prop} gives
\begin{eqnarray}\label{time.tm.pf.eq7}
m&\le  & \essinf_{\xi\in {\mathbb R}} \big(\sum_{k\in \mathbb Z} |\widehat \phi(\xi+k)|^2\big)^{1/2} \nonumber\\
& \le & \Big(\frac{1}{2\delta}\int_{-\delta}^\delta \sum_{k\in \mathbb Z} |\widehat \phi(\xi_0+\xi+k)|^2 d\xi\Big)^{1/2} \nonumber \\
& \le & \Big(\frac{1}{2\delta}\int_{-\delta}^\delta \sum_{k\in \mathbb Z}\big |\widehat \phi_N(\xi_0+\xi+k)-\widehat \phi_N(\xi_0+k)|^2 d\xi\Big)^{1/2}\nonumber\\
& &  +  \Big(\sum_{k\in \mathbb Z} |\widehat \phi_N(\xi_0+k)|^2\Big)^{1/2}+ \Big(\frac{1}{2\delta}\int_{-\delta}^\delta \sum_{k\in \mathbb Z} |(\widehat \phi-\widehat \phi_N)(\xi_0+\xi+k)|^2 d\xi\Big)^{1/2}\nonumber\\
 & \le  &  C N^{-\epsilon/4}\to 0 \ {\rm as} \ N\to \infty,
 \end{eqnarray}
which  is a contradiction.
\end{proof}

\bigskip
\subsection{Proof of Theorem \ref{fGUP}}
\begin {proof}
Note that $\phi \in L^1$ implies that $\widehat \phi$ is uniformly continuous. Now, suppose on the contrary that
\begin{equation}\label{frequency.tm.pf.eq1} |\widehat \phi(\xi)|\le C (1+|\xi|)^{-1/2-\epsilon}
\end{equation}
for some positive constants $C$ and $\epsilon>0$.
 This together with the continuity of the function $\widehat \phi$ implies that
$G_\phi(\xi)=\sum_{k\in {\mathbb Z}}|\widehat \phi(\xi+k)|^2$ is a continuous function.
Therefore there exists a positive constant $m$ such that
\begin{equation}\label{frequency.tm.pf.eq2}
G_\phi(\xi)\ge m \quad  {\rm for \ all} \ \xi\in {\mathbb R}\end{equation}
by Proposition \ref{riesz.prop} and the continuity of the function $G_\phi$.
Using the argument  in the proof of Theorem \ref{GUP}, we can find a real number $\xi_0\in {\mathbb R}$ such that
$\widehat \phi(\xi_0+k)=0$ for all $k\in {\mathbb Z}$, which implies that $G_\phi(\xi_0)=0$. This contradicts  \eqref{frequency.tm.pf.eq2}.
\end {proof}

\subsection{Proof of Theorem \ref{PointwiseFreqLoc}}
To prove Theorems \ref{PointwiseFreqLoc} and  \ref{time.tm},
 we  construct  a family of principal shift-invariant spaces on the real line which are $\nZ$-invariant for a given integer
  $n\ge 2$.
Let $g$ be an infinitely-differentiable function  that satisfies $g(x)=0$ when $x\le 0$, $g(x)=1$ when $x\ge 1$,
 and $(g(x))^2+(g(1-x))^2=1$ when $0\le x\le 1$.
For positive numbers $\alpha, \beta>0$ and a natural number $n\ge 2$, define $\psi_{\alpha, \beta, n}$ with the help of the Fourier transform by
\begin{eqnarray}\label{time.lem.eq1}
\widehat {\psi_{\alpha, \beta, n}}(\xi) & = &  h_0(\xi)+\sum_{j=1}^\infty \sum_{l=0}^{\beta_j-1}  (\beta_j)^{-1/2} h_j(\xi-n (\gamma_j+l))\nonumber\\
& & \quad +\sum_{j=1}^{\infty} \sum_{l=0}^{\beta_{j}-1} (\beta_{j})^{-1/2}
h_j(-\xi-n (\gamma_{j}+l)), 
\end{eqnarray}
where $\beta_j=\lceil2^{j\beta }\rceil$ (the smallest integer larger than or equal to  $2^{j \beta }$),
$\gamma_j=\sum_{k=0}^{j-1} \beta_{k}$, $ g_0(x)= g(x+1) g(-x+1)$, $g_1(x)=g(x+1)g(-2^\alpha x+1)$,
and
 \begin{equation}\label{time.lem.eq2}
 h_j(\xi)=\left\{\begin{array}{ll}  g_0(2\xi/(1-2^{-\alpha})) \ & \ {\rm if} \ j=0,\\
 g_1(2^{j\alpha} (2\xi-1+2^{-j\alpha})/(2^\alpha-1)) \ & \ {\rm if} \ j\ge 1.\end{array}\right.\end{equation}
The functions $\widehat {\psi_{\alpha, \beta, n}}(\xi)$  with  $\alpha=1, \beta=2$ and $n=2$
and $h_i(\xi), 0\le i\le 3$, with $\alpha=1$ are plotted in Figure \ref{figure1}.

\begin{figure}[hbt]
\centering
\begin{tabular}{c}
  \includegraphics[width=80mm]{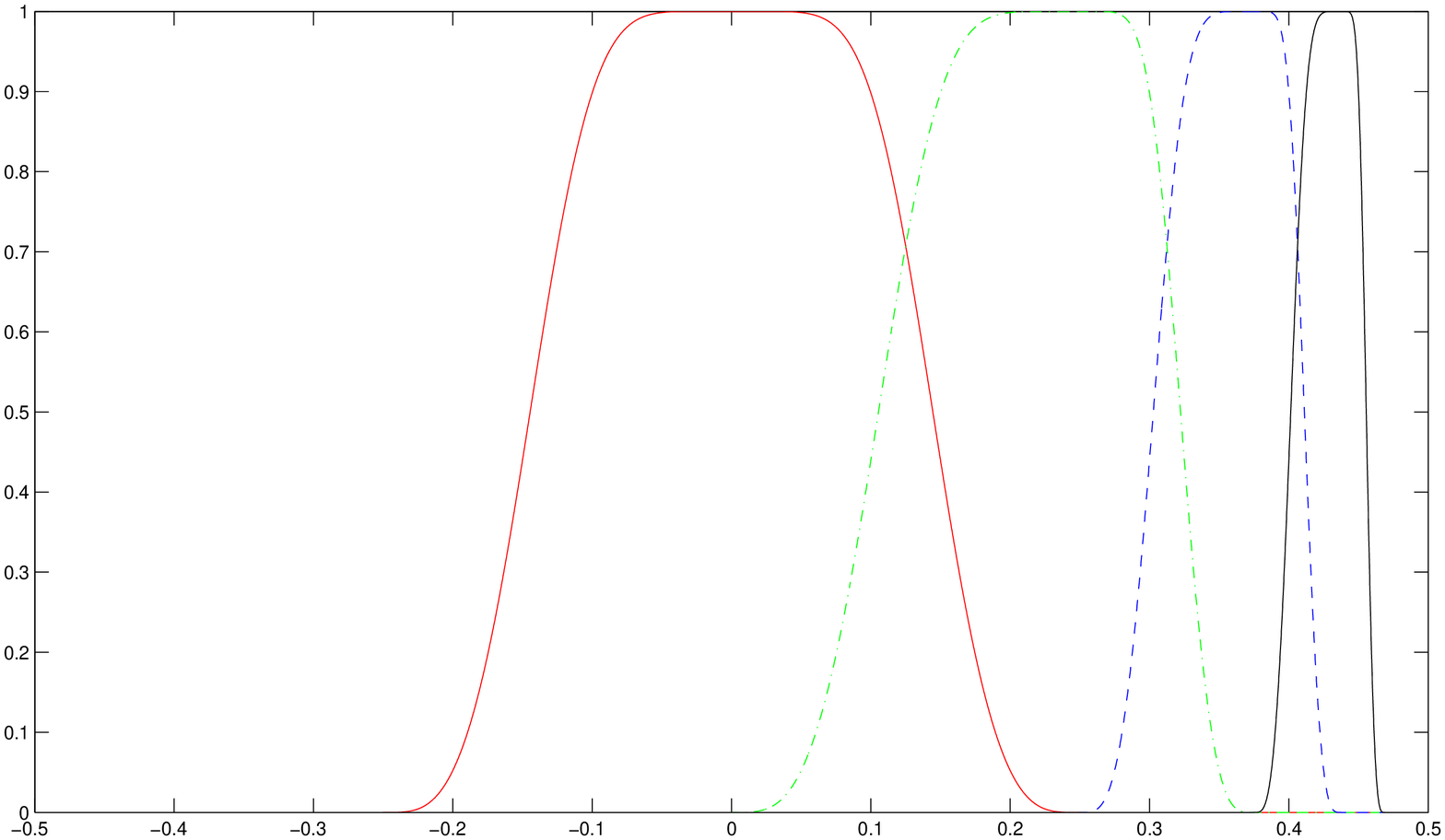} \\
    \includegraphics[width=80mm]{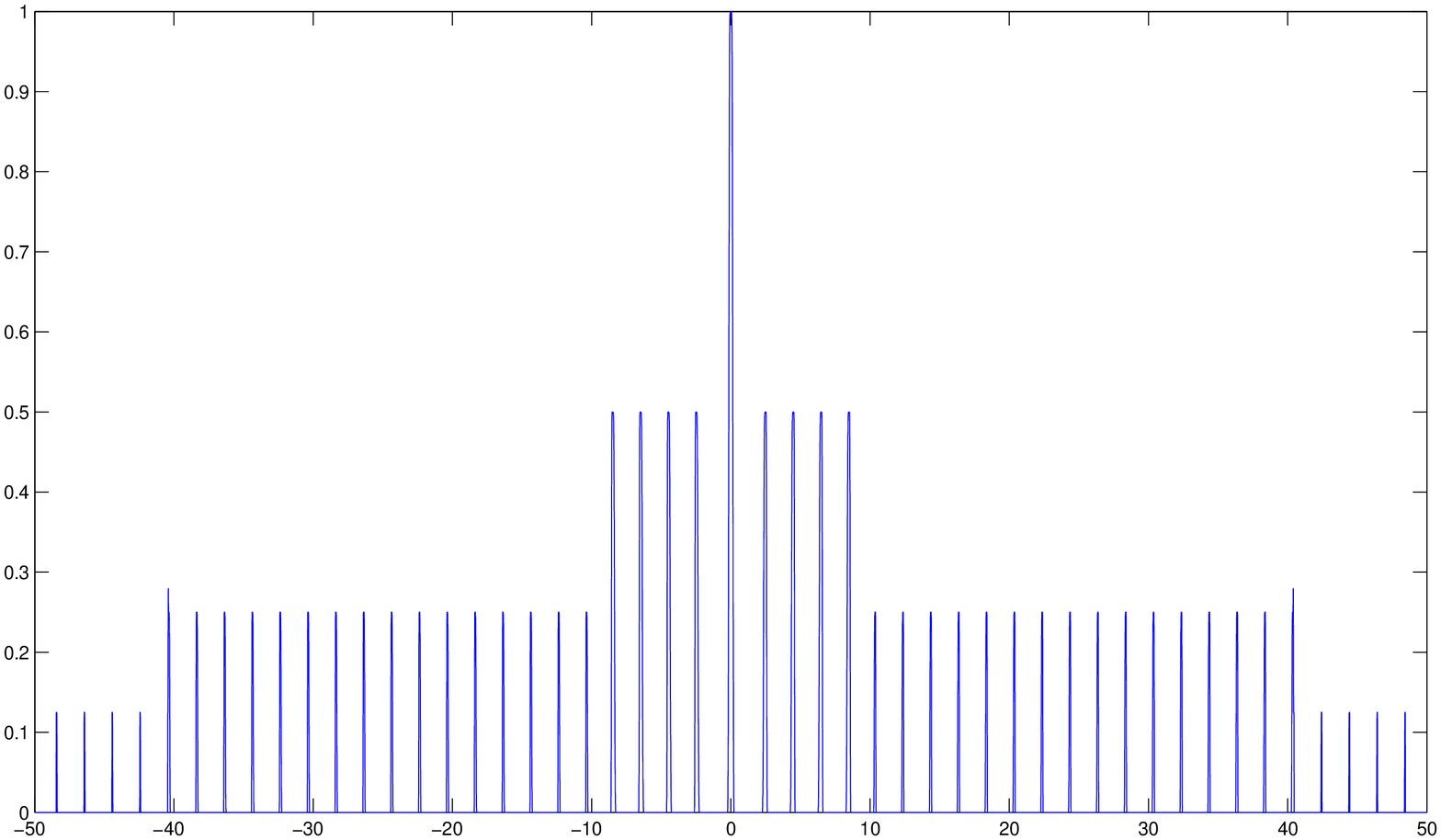}
\end{tabular}
\caption{The functions $h_i, 0\le i\le 3$ with $\alpha=1$ on the top, and the function $\widehat{\psi_{\alpha, \beta,n}}$  with $\alpha=1, \beta=2$ and $n=2$ on the bottom.
}
 \label{figure1}
\end{figure}

\begin{lemma} \label{time.lem}
For $\alpha, \beta>0$ and an integer $n\ge 2$, let $\psi_{\alpha,\beta, n}$ be defined as in \eqref{time.lem.eq1}. Then
$\psi_{\alpha, \beta, n}$ is an orthonormal generator of its generating space $V_2(\psi_{\alpha, \beta, n})$ and the principal shift-invariant space
$V_2(\psi_{\alpha, \beta,n})$ is $\nZ$-invariant.
\end{lemma}

\begin{proof}
As each $h_j$, for $j\ge 0$, is supported in $(-1/2, 1/2)$ by  construction,
\begin{eqnarray*}
|\widehat {\psi_{\alpha, \beta, n}}(\xi)|^2 & = &  |h_0(\xi)|^2+\sum_{j=1}^\infty \sum_{l=0}^{\beta_j-1}  (\beta_j)^{-1} |h_j(\xi-n (\gamma_j+l))|^2\nonumber\\
& & \quad +\sum_{j=1}^{\infty} \sum_{l=0}^{\beta_{j}-1} (\beta_{j})^{-1}
|h_j(-\xi-n (\gamma_{j}+l))|^2,
\end{eqnarray*}
which implies that
\begin{eqnarray}\label{time.lem.pf.eq1}
\sum_{k\in \mathbb Z} |\widehat {\psi_{\alpha, \beta, n}}(\xi+k)|^2 & = &
\sum_{k\in \mathbb Z} |h_0(\xi+k)|^2+\sum_{j=1}^\infty
 \sum_{k\in \mathbb Z} \big( |h_j(\xi+k)|^2+ |h_j(-\xi+k)|^2\big)\nonumber\\
 & = & |h_0(\xi)|^2+\sum_{j=1}^\infty
 |h_j(\xi)|^2+\sum_{j=1}^\infty
 |h_j(-\xi)|^2
  \end{eqnarray}
 for any $\xi\in (-1/2, 1/2)$.
Set
$$H(\xi):=|h_0(\xi)|^2+\sum_{j=1}^\infty
 |h_j(\xi)|^2+\sum_{j=1}^\infty
 |h_j(-\xi)|^2.$$
 Then $H(\xi)$ is a symmetric function supported on $(-1/2, 1/2)$ and
 for any
 $\xi\in [1-2^{-j\alpha},  1-2^{-(j+1)\alpha}]/2$ with $ j\ge 0$,
  \begin{eqnarray}
 H(\xi) & = &  |h_{j}(\xi)|^2+ |h_{j+1}(\xi)|^2\nonumber\\
  & = & | g(- 2^{(j+1)\alpha} (2\xi-1+2^{-j\alpha})/(2^\alpha-1)+1 )|^2\nonumber\\
 & &   +
 | g(2^{(j+1)\alpha} (2\xi-1+2^{-(j+1)\alpha})/(2^\alpha-1)+1 )|^2\nonumber\\
& = & 1 \end{eqnarray}
by the construction of the functions $g$ and $h_j, j\ge 0$.
Therefore $H(\xi)=1$ for all $\xi\in (-1/2, 1/2)$, which together with \eqref{time.lem.pf.eq1} implies that
\begin{equation}
\label{time.lem.pf.eq2}
\sum_{k\in \mathbb Z} |\widehat {\psi_{\alpha, \beta, n}}(\xi+k)|^2=1\quad {\rm for\ all}\ \xi\in {\mathbb R}\backslash (1/2+{\mathbb Z}).
\end{equation}
Then   $\psi_{\alpha, \beta,n}$ is an orthonormal generator for its generating space $V_2(\psi_{\alpha, \beta, n})$
by \eqref{time.lem.pf.eq2} and Proposition \ref{riesz.prop}.

By \eqref{time.lem.eq1}, $\widehat {\psi_{\alpha, \beta, n}}$ is supported on $(-1/2, 1/2)+n{\mathbb Z}$.
Then $V_2(\psi_{\alpha, \beta,n})$ is $\nZ$-invariant by \eqref{time.lem.pf.eq2} and Proposition \ref{sis.prop2}.
\end{proof}
 We are now ready to prove Theorem \ref {PointwiseFreqLoc}.

 \begin {proof} [Proof of Theorem \ref {PointwiseFreqLoc}]
Let $\psi_{\alpha, \beta, n}$ be as in \eqref{time.lem.eq1} for $\alpha,\beta>0$, and set $\phi=\psi_{\alpha, \beta, n}$. Then by Lemma \ref{time.lem} it suffices to prove \eqref{frequency.tm.eq2} for the function $\phi$ just defined.
From \eqref{time.lem.eq1} it follows that
\begin{eqnarray}
 & & |\widehat \phi(\xi)| |\xi|^{1/2}  =  |\widehat {\psi_{\alpha, \beta, n}}(\xi)| |\xi|^{1/2}\nonumber\\
 & \le &
\sup \Big\{ |h_0(\xi)| |\xi|^{1/2}, \sup_{j\ge 1, 0\le l\le \beta_j-1}
\beta_{j}^{-1/2} |h_j(\xi-n(\gamma_j+l))| |\xi|^{1/2},\nonumber\\
& &
\sup_{j\ge 1, 0\le l\le \beta_{j}-1}
\beta_{j}^{-1/2} |h_{j}(-\xi-n(\gamma_{j}+l))| |\xi|^{1/2}\Big\}\nonumber
\end{eqnarray}
Note that, from its definition, $h_{j}(\xi-n(\gamma_{j}+l))$ has support in $[n(\gamma_{j}+l), n(\gamma_{j}+l)+1]$ and has maximal value $1$. Thus the term  $|h_j(\xi-n(\gamma_j+l))| |\xi|^{1/2}$ can be bounded above by $\big(n(\gamma_j+l)+1\big)^{1/2}$ for all $\xi$ and $0\le l\le \beta_j-1$.  Thus, it follows from the last inequality  and the fact that $\gamma_j+\beta_j=\gamma_{j+1}$ that

\begin{equation} |\widehat \phi(\xi)| |\xi|^{1/2}\le  1+ C \sup_{j\ge 1} (\beta_j)^{-1/2} (\gamma_{j+1})^{1/2}<\infty,
\end {equation}
where $C$ is a positive constant. Hence  \eqref{frequency.tm.eq2} holds. In particular, we can show that
\begin{equation}
0<\limsup_{|\xi|\to \infty} |\widehat \phi(\xi)| |\xi|^{1/2}<\infty.
\end{equation}
This proves the pointwise frequency localization of the theorem.  The time localization inequality is a direct consequence of Lemma \ref {TimeLem} below. The fact that $\phi$ is also in $L^1$ follows from Lemma \ref  {TimeLemp} choosing $p=1$, $\gamma=0$.
 \end {proof}

\subsection{Proof of Theorem \ref{time.tm}}
 Theorem \ref{time.tm} is an immediate consequence of the following three lemmas:
\begin {lemma}
\label {TimeLem}
Let $\epsilon\in (0,1)$, $\alpha, \beta>0$, $n$ be an integer with $n\ge 2$, and $\psi_{\alpha,\beta, n}$ be defined as in \eqref{time.lem.eq1}.
Then
\begin {equation}
\label {TimeLemEq}
\int_{\mathbb R}   |\psi_{\alpha,\beta, n}(x)|^2 |x|^{1-\epsilon}dx<\infty.
\end {equation}
\end{lemma}
\begin {lemma}
\label {TimeLemp}
Let $\gamma\ge 0$, $1\le p<2,$ $n$ be an integer with $n\ge 2$, and $\psi_{\alpha, \beta, n}$ be defined as in \eqref{time.lem.eq1} for
  positive numbers $\alpha, \beta>0$  with $\beta (1/p-1/2)+\alpha (p-1-\gamma)/p>0.$
Then
\begin {equation}
\label {TimeLempEq}
\int_{\mathbb R}   |\psi_{\alpha,\beta, n}(x)|^p(1+ |x|)^{\gamma}dx<\infty.
\end {equation}
\end{lemma}
\begin{lemma}
\label {FreqLem}
Let $\delta>0$, $1\le q<\infty$, $n$ be an integer with $n\ge 2$, and $\psi_{\alpha, \beta, n}$ be defined as in \eqref{time.lem.eq1} for
  positive numbers $\alpha, \beta>0$   with $\alpha>\beta (1+\delta-q/2)$. Then
\begin {equation}
\label {FreqLemEq}
\int_{\mathbb R}   |\widehat {\psi_{\alpha, \beta, n}}(\xi)|^q (1+|\xi|)^{\delta} d\xi<\infty.
\end {equation}
\end{lemma}
\begin{proof} [Proof of Lemma \ref {TimeLem}]
Taking the inverse Fourier transform of both sides of \eqref{time.lem.eq1}  yields
 \begin{eqnarray}\label{time.tm.pf.eq14}
 \psi_{\alpha, \beta, n}(x) & = & \frac{1-2^{-\alpha}}{2} g_0^\vee\big(\frac{1-2^{-\alpha}}{2} x\big)+\frac{2^\alpha-1}{2}
 \sum_{j=1}^\infty (\beta_j)^{-1/2} 2^{-j\alpha} g_1^\vee\big(\frac{2^\alpha-1}{2^{j\alpha+1}}x\big) \nonumber\\
 & & \quad \times e^{\pi i x (1-2^{-j\alpha})}     \Big(\sum_{l=0}^{\beta_j-1} e^{2\pi i x n(\gamma_j+l)}\Big)
 + \frac{2^\alpha-1}{2}
 \sum_{j=1}^{\infty} (\beta_{j})^{-1/2} 2^{-j\alpha}\nonumber\\
 & & \quad \times g_1^\vee\big(-\frac{2^\alpha-1}{2^{j\alpha+1}}x\big)
 \times e^{-\pi i x (1-2^{-j\alpha})}     \Big(\sum_{l=0}^{\beta_{j}-1} e^{-2\pi i x n(\gamma_{j}+l)}\Big),
 \end{eqnarray}
 where $g_0^\vee$ and $g_1^\vee $ are the inverse Fourier transforms of the functions $g_0$ and $g_1$ respectively.
Since both $g_0$ and $g_1$ are compactly supported and infinitely differentiable,
their inverse Fourier transforms $g_0^\vee$ and $g_1^\vee $ have polynomial decay at infinity. In particular
\begin{equation*}
|g_0^\vee (x)|+ | g_1^\vee (x)|\le C (1+|x|)^{-2}, \ x\in {\mathbb R}
\end{equation*}
for some positive constant $C$.
Hence
 \begin{eqnarray}\label{time.tm.pf.eq15}
 & &
 \Big(\int_{\mathbb R} |\psi_{\alpha, \beta,n}(x)|^2 (1+|x|)^{1-\epsilon} dx\Big)^{1/2}\nonumber \le
 \Big(\frac{1-2^{-\alpha}}{2}\Big)
  \Big(\int_{\mathbb R} |g_0^\vee(x)|^2 (1+|x|)^{1-\epsilon} dx\Big)^{1/2}\nonumber\\
  & & \quad \quad+ (2^\alpha-1)
 \sum_{j=1}^\infty (\beta_j)^{-1/2} 2^{-j\alpha}\Big (\int_{\mathbb R}\Big|g_1^\vee\big(\frac{2^\alpha-1}{2^{j\alpha+1}}x\big)\Big|^2
 \Big(\frac{\sin  \beta_j n \pi x}{\sin n\pi  x}\Big)^2  (1+ |x|)^{1-\epsilon} dx\Big)^{1/2}\nonumber\\
 & \le & C + C \sum_{j=1}^\infty 2^{-j(\beta+\alpha+\alpha \epsilon)/2} \Big(\int_{\mathbb R} (1+ 2^{-j\alpha} |x|)^{-2}
 \Big(\frac{\sin   \beta_j\pi x}{\sin \pi  x}\Big)^2   dx\Big)^{1/2}\nonumber\\
 & = & C + C \sum_{j=1}^\infty 2^{-j(\beta+\alpha+\alpha\epsilon)/2} \Big(\int_{-1/2}^{1/2} \Big(\sum_{l\in {\mathbb Z}} (1+ 2^{-j\alpha} |x+l|)^{-2}\Big)
 \Big(\frac{\sin   \beta_j\pi x}{\sin \pi  x}\Big)^2   dx\Big)^{1/2}\nonumber\\
& \le & C + C \sum_{j=1}^\infty  2^{-j(\beta/2+\alpha\epsilon/2)}\Big(\int_{-1/2}^{1/2}
 \Big(\frac{\sin  \beta_j \pi x}{\sin \pi x}\Big)^2   dx\Big)^{1/2}\nonumber\\
 & \le &  C + C \sum_{j=1}^\infty  2^{-j(\beta+\alpha\epsilon)/2}\Big(  \int_{-1/2}^{1/2}
 \big(\min(\beta_j, \frac{1}{2|x|})\big)^2  dx\Big)^{1/2}\nonumber\\
  & \le &
  C + C \sum_{j=1}^\infty  2^{-j\alpha\epsilon/2} <\infty,  \end{eqnarray}
where $C$ is a positive constant which could be different at different occurrences.
 \end{proof}
\begin{proof} [Proof of Lemma \ref {TimeLemp}]
Similar to the argument in Lemma \ref{TimeLem} we have
 \begin{eqnarray*}
 & &\Big(\int_{\mathbb R} |\psi_{\alpha, \beta,n}(x)|^p (1+|x|)^{\gamma} dx\Big)^{1/p}\nonumber \\
& \le &    C+  C \sum_{j=1}^\infty  2^{-j(\beta/2+\alpha(1-(1+\gamma)/p))}\Big(\int_{-1/2}^{1/2}
 \Big(\frac{\sin  \beta_j \pi x}{\sin \pi  x}\Big)^p  dx\Big)^{1/p}\nonumber\\
 &\le & \left\{\begin{array} {ll}
  C+ C \sum_{j=1}^\infty  2^{-j(\beta(1/p-1/2)+\alpha(p-1-\gamma)/p)} & {\rm if} \ 1<p<2\\
  C+C \sum_{j=1}^\infty  j 2^{-j(\beta/2-\gamma\alpha)}  & {\rm if} \ p=1\end{array}\right.\nonumber\\
&  < &\infty,
  \end{eqnarray*}
from which the lemma follows.
\end{proof}

\begin {proof} [Proof of Lemma \ref {FreqLem}]
By \eqref{time.lem.eq1}, we have
\begin{eqnarray}
 & &\int_{\mathbb R} |\widehat {\psi_{\alpha, \beta, n}}(\xi)|^q (1+|\xi|)^\delta d\xi\nonumber\\
 & = &
\int_{\mathbb R} |h_0(\xi)|^q  (1+|\xi|)^\delta d\xi   +
 \sum_{j= 1}^\infty\sum_{l=0}^{\beta_j-1}
\beta_{j}^{-q/2} \int_{\mathbb R} |h_j(\xi-n(\gamma_j+l))|^q (1+|\xi|)^{\delta} d\xi\nonumber\\
 & & +   \sum_{j=1}^\infty\sum_{l=0}^{\beta_j-1} \beta_{j}^{-q/2} \int_{\mathbb R} |h_{j}(-\xi-n(\gamma_{j}+l))|^q (1+|\xi|)^{\delta} d\xi\nonumber\\
& \le & C   + C
 \sum_{j= 1}^\infty\sum_{l=0}^{\beta_j-1}
2^{-\beta j(q/2-\delta)} \int_{\mathbb R} |h_j(\xi-n(\gamma_j+l))|^q  d\xi\nonumber\\
 & & + C  \sum_{j=1}^\infty\sum_{l=0}^{\beta_j-1} 2^{-\beta j(q/2-\delta)} \int_{\mathbb R} |h_{j}(-\xi-n(\gamma_{j}+l))|^q  d\xi\nonumber\\
 &\le  & C+ C \sum_{j\ge 1} 2^{j\beta(\delta-q/2+1)-\alpha j}<\infty,
\end{eqnarray}
where $C$ is a positive constant  which could be different at different occurrences.
Hence  the lemma  is established.
\end{proof}

\bigskip

{\bf Acknowledgement} \quad The first named author would like to thank Professors Christopher Heil, Palle Jorgensen, and Gestur Olafsson for their stimulating discussions.  The second named author contributed to this project while visiting the  Department of Mathematics, at Vanderbilt University on his sabbatical during the fall of 2009. The second named author would like to thank Professor Akram Aldroubi and the department for the hospitality.

\end{document}